\documentclass{article}
\usepackage{amssymb,amsmath,amsthm,graphicx,kotex} 

\def\qed{\hfill {\hbox{${\vcenter{\vbox{               
   \hrule height 0.4pt\hbox{\vrule width 0.4pt height 6pt
   \kern5pt\vrule width 0.4pt}\hrule height 0.4pt}}}$}}}

\def\utr{\underline\triangleright}
\def\otr{\overline\triangleright}

\def\bar{\overline}

\newtheorem{theorem}{Theorem}
\newtheorem{lemma}[theorem]{Lemma}

\theoremstyle{definition}
\newtheorem{example}{Example}
\newtheorem{definition}{Definition}
\newtheorem{remark}{Remark}

\date{}

\title{\Large \textbf{Skew Brace Enhancements and Virtual Links}}

\author{
Melody Chang\footnote{Email: mchang2353@scrippscollege.edu.}
 \and
Sam Nelson\footnote{Email: Sam.Nelson@cmc.edu. Partially supported by
Simons Foundation Collaboration Grant 702597.}}

\begin{document}
\maketitle

\begin{abstract}
We use the structure of skew braces to enhance the biquandle counting invariant
for virtual knots and links for finite biquandles defined from 
skew braces. We introduce two new invariants: a single-variable polynomial
using skew brace ideals and a two-variable polynomial using the skew brace group
structures. We provide examples to show that the new invariants are not 
determined by the counting invariant and hence are proper enhancements.
\end{abstract}

\parbox{6in} {\textsc{Keywords:} Skew braces, biquandles, enhancements of 
counting invariants, virtual knots

\smallskip

}

\section{\large\textbf{Introduction}}\label{I}

\textit{Skew braces} are a type of algebraic structure consisting of a set
with two group operations (analogous to a ring) which interact via a kind of 
modified distributive law. They have been studied for the last decade or so 
in papers such as \cite{B,GV,KSV,R}. Beyond their inherent algebraic interest,
skew braces are important in knot theory because they provide solutions to 
the set-theoretic Yang-Baxter equation, which in turn lead to invariants of 
knots.

Set-theoretic Yang-Baxter solutions are also provided by \textit{biquandles},
which have been studied for about the last two decades; see \cite{EN} and the
references therein for more. Biquandle-based knot invariants include many 
examples of \textit{enhancements}, invariants which specialize to the 
integer-valued biquandle counting invariant (i.e., the cardinality 
of the set of biquandle homomorphisms from the fundamental biquandle of the 
knot or link to a fixed finite biquandle) but which contain more information 
beyond merely the set's cardinality. 

A skew brace defines a biquandle, and thus the notion of biquandle coloring 
extends to a notion of skew brace coloring, allowing for skew brace counting
invariants and enhancements thereof. Like many biquandle-based invariants of 
classical knots and links, these skew brace invariants extend to the setting 
of virtual knot theory in a natural way by simply ignoring the virtual 
crossings, i.e., letting skew brace colors remain constant when passing 
through virtual crossings.

In this paper we define two infinite families of new polynomial
enhancements of the biquandle counting invariant for biquandles which come 
from skew braces. The paper is organized as follows. In Section \ref{QB} we 
collect a few definitions and observations about skew braces and their 
relationship with biquandles. We recall the biquandle counting invariant 
and define the skew brace counting invariant for finite biquandles and skew 
braces.
In Section \ref{GF} we introduce the new enhanced invariants: the two-variable
\textit{skew-brace enhanced polynomial} and the one-variable \textit{skew 
brace ideal polynomial}. In Section \ref{E} we provide examples to illustrate 
the computation of the invariant and to establish that the new invariants are 
proper enhancements, i.e. that they are not determined by the counting 
invariant. In Section \ref{Q} we end with some questions for future research.

\section{\large\textbf{Skew Braces and Skew Brace Colorings}}\label{QB}

We begin with a definition; see \cite{B,KSV,GV,R} for more.

\begin{definition}
A \textit{skew brace} is a set $X$ with two group operations which we
will denote by $(x,y)\mapsto x\circ y$ and $(x,y)\mapsto x*y$ with 
inverses denoted by $x^{\circ}$ and $x^*$, satisfying a modified
distributive law
\[x\circ (y*z)=(x\circ y)*x^**(x\circ z).\]
\end{definition}

To familiarize ourselves a little with the algebra of skew braces, let us
note the standard fact (see \cite{B} for example) that this modified 
distributive law has the useful consequence that the two group identities 
are equal.

\begin{lemma}
Let $X$ be a skew brace. Then the identity elements with respect to both
operations are the same.
\end{lemma}

\begin{proof}
Let us temporarily denote the $*$-identity element by $e_*$ and the 
$\circ$-identity element by $e_{\circ}$. Then we observe that for any $x,y\in X$ 
we have
\[x\circ y  =  x\circ (e_**y)  =  (x\circ e_*)*x^**(x\circ y)\]
and thus
\[(x\circ y)*(x\circ y)^* = (x\circ e_*)*x^**(x\circ y)*(x\circ y)^*\]
so we have 
\[e_* =(x\circ e_*)*x^*\]
which implies \[x=x\circ e_*.\] Then
\[x^{\circ}\circ x=x^{\circ}\circ x\circ e_*\]
which says $e_{\circ}=e_*$ as required.
\end{proof}

For simplicity, we will denote the common identity element of both operations
as $e$.

Let $X$ be a skew brace and let $K$ be an oriented classical or virtual
knot or link diagram.
A \textit{skew brace coloring} of $K$, also called an 
\textit{$X$-coloring of $K$}, is an assignment of elements in $X$ to each 
semiarc in $K$ such that at every classical crossing we have one of
the following pictures:
\[\includegraphics{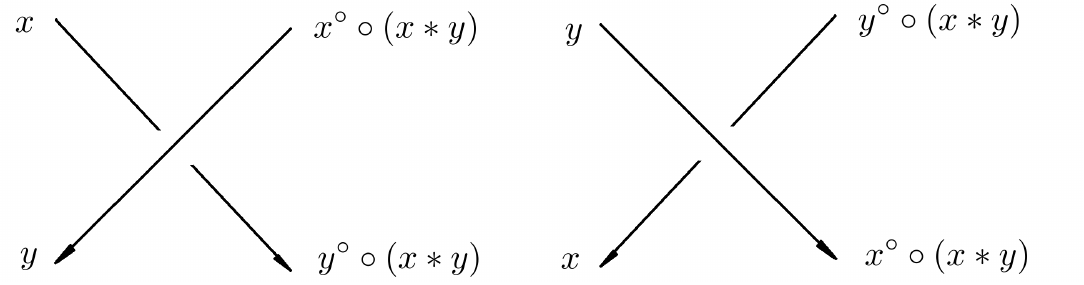}\]
At virtual crossings, the colors on the semiarcs are not changed.

Next we recall a definition found in \cite{EN}:
\begin{definition}
A \textit{biquandle} is a set $X$ with binary operations 
$\utr,\otr$ satisfying
\begin{itemize}
\item[(i)] For all $x\in X$,
\[x\utr x=x\otr x\]
\item[(ii)] For all $x,y\in X$, the operations $\utr$ and $\otr$ are 
right-invertible and the map of pairs $S(x,y)=(y\otr x,x\utr y)$ is invertible,
and
\item[(iii)] For all $x,y,z\in X$ the \textit{exchange laws}
\[\begin{array}{rcl}
(x\utr y)\utr(z\utr y) & = & (x\utr z)\utr(y\otr z) \\
(x\utr y)\otr(z\utr y) & = & (x\otr z)\utr(y\otr z) \\
(x\otr y)\otr(z\otr y) & = & (x\otr z)\otr(y\utr z) \\
\end{array}\]
\end{itemize}
are satisfied.
\end{definition}

As noted in \cite{GV}, skew brace colorings provide solutions to the 
set-theoretic Yang Baxter equation. In fact, we have

\begin{theorem}
A skew brace is a biquandle under the operations
\[x\utr y= y^{\circ}\circ(x*y)\quad \mathrm{and}\quad 
x\otr y= y^{\circ}\circ(y*x)\]
\end{theorem}

The result has been previously established in Corollary 3.3 of \cite{SV}
using different notation; let us verify using our notation.

\begin{proof}
We must verify that the biquandle axioms are satisfied. Checking, we have
\[x\utr x=x^{\circ}\circ(x*x)=x\otr x\] 
as required, and axiom (i) is satisfied.

For axiom (ii), we note that  the operations
\[x\utr^{-1} y= (y\circ x)*y^*\quad \mathrm{and}\quad 
x\otr^{-1} y= y^**(y\circ x)\]
are right-inverses for $\utr$ and $\otr$:
\begin{eqnarray*}
(x\utr^{-1} y)\utr y & = & y^{\circ}\circ((y\circ x)*y^**y) = x,\\
(x\utr y)\utr^{-1}y & = &(y\circ y^{\circ}\circ(x*y))*y^*=x,\\
(x\otr^{-1} y)\otr y & = & y^{\circ}\circ(y*(y^**(y\circ x))) =x\quad \mathrm{and} \\
(x\otr y)\otr^{-1} y & = & y^**(y\circ (y^{\circ}\circ(y*x))) =x
\end{eqnarray*}
and that the map $S^{-1}:X\times X\to X\times X$ given by
\[S^{-1}(x,y)=(((x\circ y^{\circ})*x^*)^{\circ},((x\circ y^{\circ})*x^*)^{\circ}\circ x\circ y^{\circ})\]
is the inverse of the map 
$S(x,y)=(x^{\circ}\circ(x*y),y^{\circ}\circ(x*y))=(y\otr x,x\utr y)$.

Now, let us consider the exchange laws. We compute
\begin{eqnarray*}
(x\utr y)\utr(z\utr y) 
& = & [y^{\circ}\circ(x*y)]\utr [y^{\circ}\circ (z*y)] \\
& = & [y^{\circ}\circ (z*y)]^{\circ}\circ[(y^{\circ}\circ(x*y))*(y^{\circ}\circ (z*y))] \\
& = & [y^{\circ}\circ (z*y)]^{\circ}\circ[y^{\circ}\circ(x*y)]*[y^{\circ}\circ (z*y)]^{\circ *}*[(y^{\circ}\circ (z*y))^{\circ}\circ(y^{\circ}\circ (z*y))] \\
& = & [y^{\circ}\circ (z*y)]^{\circ}\circ[y^{\circ}\circ(x*y)]*[y^{\circ}\circ (z*y)]^{\circ *} \\
& = & [y^{\circ}\circ (z*y)]^{\circ}\circ[y^{\circ}\circ(x*y)]*[(z*y)^{\circ}\circ y]^* \\
& = & [(z*y)^{\circ}\circ y\circ y^{\circ}\circ(x*y)]*[(z*y)^{\circ}\circ y]^* \\
& = & [(z*y)^{\circ}\circ(x*y)]*[(z*y)^{\circ}\circ y]^* \\
& = & ((z*y)^{\circ}\circ x)*(z*y)^{\circ *}*[(z*y)^{\circ}\circ y]*[(z*y)^{\circ}\circ y]^* \\
& = & ((z*y)^{\circ}\circ x)*(z*y)^{\circ *} 
\end{eqnarray*}
and
\begin{eqnarray*}
(x\utr z)\utr(y\otr z)
& = & [z^{\circ} \circ(x*z)]\utr[z^{\circ}\circ(z*y)] \\
& = & [z^{\circ}\circ(z*y)]^{\circ}\circ[(z^{\circ} \circ(x*z))*(z^{\circ}\circ(z*y))] \\
& = & [(z^{\circ}\circ(z*y))^{\circ}\circ((z^{\circ} \circ(x*z))]*[z^{\circ}\circ(z*y)]^{\circ *}*[(z^{\circ}\circ(z*y))^\circ \circ(z^{\circ}\circ(z*y))] \\
& = & [(z^{\circ}\circ(z*y))^{\circ}\circ((z^{\circ} \circ(x*z))]*[z^{\circ}\circ(z*y)]^{\circ *} \\
& = & [(z*y)^{\circ}\circ z\circ z^{\circ} \circ(x*z)]*[z^{\circ}\circ(z*y)]^{\circ *} \\
& = & [(z*y)^{\circ}\circ(x*z)]*[(z*y)^{\circ}\circ z]^* \\
& = & ((z*y)^{\circ}\circ x)*(z*y)^{\circ *}*[(z*y)^{\circ}\circ z)]*[(z*y)^{\circ}\circ z]^* \\
& = & ((z*y)^{\circ}\circ x)*(z*y)^{\circ *} 
\end{eqnarray*}
and the first exchange law is satisfied. The other two are similar and left 
to the reader.
\end{proof}

A skew brace coloring is a biquandle coloring by the associated
biquandle of the skew brace.
Hence, it follows that the number of skew brace colorings of an oriented 
classical or virtual knot or link 
diagram $K$ by a finite skew brace $X$ is a knot invariant, which we will 
denote by $\Phi_X^{\mathbb{Z}}(K)=|\mathcal{C}(K,X)|$ where $\mathcal{C}(K,X)$ 
denotes the set of $X$-colorings of $K$.

The set-theoretic Yang-Baxter solutions defined by skew braces with commutative
$*$ operation are \textit{involutive}, meaning that the vertical map of pairs
\[r(x,y)=(x^**(x\circ y),\ (x^**(x*y))^{\circ}\circ x\circ y)\]
satisfies $r^2=\mathrm{Id}$. More precisely, we note that the first component
of $r^2(x,y)$ is
\begin{eqnarray*}
(x^**(x\circ y))^**[(x^**(x\circ y)) \circ(x^**(x*y))^{\circ}\circ x\circ y)] 
& = & (x^**(x\circ y))^**[x\circ y] \\
& = & (x\circ y)^**(x^*)^**(x\circ y) \\
& = & (x\circ y)^**x*(x\circ y) 
\end{eqnarray*}
which equals $x$ if $*$ is commutative, and
the second component is
\begin{eqnarray*}
[(x\circ y)^**x*(x\circ y)]^{\circ}\circ(x^**(x\circ y))\circ(x^**(x\circ y))^{\circ}\circ x\circ y 
& = & [(x\circ y)^**x*(x\circ y)]^{\circ} \circ x\circ y 
\end{eqnarray*}
which again reduces to $y$ in the case that $*$ is commutative.

\begin{remark}
Various notational conventions for the skew brace operations are used in the 
literature, including commonly writing $+$ for $*$ and $-x$ for $x^*$ even 
when $*$ is noncommutative. To avoid certain errors, e.g. drawing the 
conclusion that all skew braces are involutive, we will prefer to use the 
$*$ notation.
\end{remark}

The fact that all groups of cardinality less than six are abelian 
implies that the coloring invariant and its enhancements for skew braces with
fewer than six elements cannot 
distinguish between knots or links related by the \textit{2-move}:
\[\includegraphics{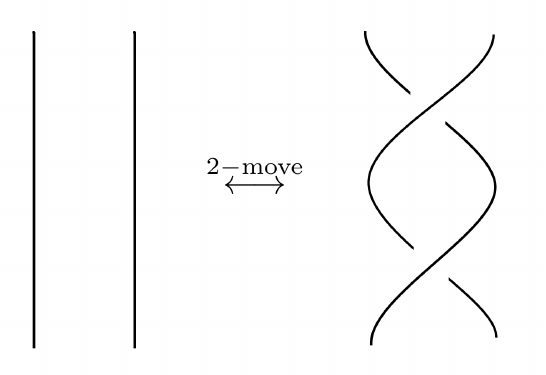}.\]
Since the 2-move can be combined with a Reidemeister II move to yield a crossing
change, 
\[\includegraphics{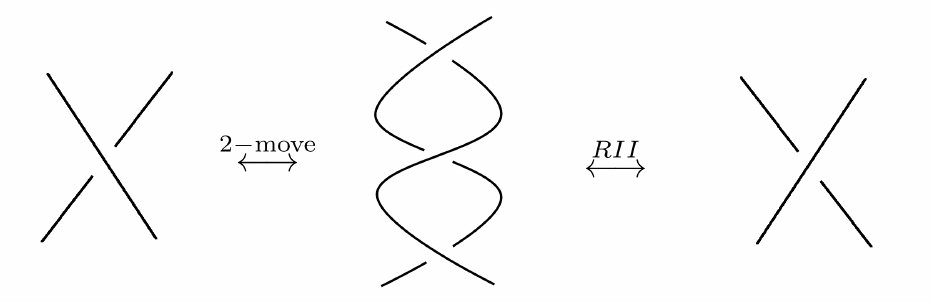}\]
it follows that knots and links in various categories (classical, virtual, 
flat virtual, etc.) which are related by crossing change cannot be 
distinguished by involutory skew brace invariants. Indeed, involutory
skew brace counting invariants and their enhancements are trivial for classical
knots and links. However, virtual knots are links fall into several distinct
classes under the virtual Reidemeister moves together with crossing changes;
as we will show, involutory skew brace invariants and their enhancements can 
be effective at distinguishing these classes of virtual knots and links. 
Moreover, involutory skew brace invariants are well-defined for flat virtual 
knots.

Finite skew brace structures on a set $X$ can be specified in various ways --
algebraic formulas for the two group operations, using pairs of tuples as 
described in \cite{KSV}, etc. For our purposes it will be most useful
to specify a skew brace on $X=\{1,2,\dots, n\}$ with a pair of operation
tables for the two group operations, which we will call \textit{structure 
tables}.

\begin{example}
The structure tables
\[
\begin{array}{r|rrrr}
\circ & 1 & 2 & 3 & 3 \\ \hline
1 & 1 & 2 & 3 & 4 \\
2 & 2 & 1 & 4 & 3 \\
3 & 3 & 4 & 1 & 2 \\
4 & 4 & 3 & 2 & 1
\end{array}\quad
\begin{array}{r|rrrr}
* & 1 & 2 & 3 & 3 \\ \hline
1 & 1 & 2 & 3 & 4 \\
2 & 2 & 3 & 4 & 1 \\
3 & 3 & 4 & 1 & 2 \\
4 & 4 & 1 & 2 & 3
\end{array}
\]
define a skew brace with $\circ$-group isomorphic to the Klein 4-group 
$\mathbb{Z}_2\oplus\mathbb{Z}_2$ and $*$-group isomorphic to $\mathbb{Z}_4$.
\end{example}

If $\ast$ is nonabelian, then the skew brace's resulting set-theoretic 
Yang-Baxter solution $r$ may not be involutive; these are the examples which
can give us interesting invariants of classical knots and links.
\begin{example}\label{ex:nab}
The skew brace structure on the set $X=\{1,2,3,4,5,6\}$ defined by the structre
tables
\[
\begin{array}{r|rrrrrr}
\circ & 1 & 2 & 3 & 4 & 5 & 6 \\ \hline
1 & 1 & 2 & 3 & 4 & 5 & 6 \\
2 & 2 & 3 & 1 & 5 & 6 & 4 \\
3 & 3 & 1 & 2 & 6 & 4 & 5 \\
4 & 4 & 5 & 6 & 3 & 1 & 2 \\
5 & 5 & 6 & 4 & 1 & 2 & 3 \\
6 & 6 & 4 & 5 & 2 & 3 & 1
\end{array}
\quad 
\begin{array}{r|rrrrrr}
* & 1 & 2 & 3 & 4 & 5 & 6 \\ \hline
1 & 1 & 2 & 3 & 4 & 5 & 6 \\
2 & 2 & 3 & 1 & 5 & 6 & 4 \\
3 & 3 & 1 & 2 & 6 & 4 & 5 \\
4 & 4 & 6 & 5 & 1 & 3 & 2 \\
5 & 5 & 4 & 6 & 2 & 1 & 3 \\
6 & 6 & 5 & 4 & 3 & 2 & 1
\end{array}
\]
defines a non-involutive set-theoretic Yang-Baxter solution with for example
$r(4,3)=(2,5)\ne(4,3)$.
\end{example}

\section{\large\textbf{Skew Brace Enhancements}}\label{GF}

In this section we introduce new invariants of oriented  virtual knots and 
links which enhance the skew brace counting invariant.
More precisely, these are polynomial invariants which specialize to the
counting invariant when the variables are set equal to 1. 

Recall that for any biquandle coloring $f\in\mathcal{C}(X,K)$, the 
\textit{image} of $f$, denoted $\mathrm{Im}(f)$, is the biquandle closure of 
the set of biquandle elements used in $f$, i.e, the set of elements of $X$ 
obtainable from biquandle elements used in $f$ using the biquandle operations 
$\utr$ and $\otr$ (as well as their right inverses, though for finite 
biquandles the right inverse operations can be expressed in terms of the 
standard operations). Let us denote the group closures of $S\subset X$ under 
the group operations $*$ and $\circ$ by $\bar{S^*}$ and $\bar{S^{\circ}}$ 
respectively.

\begin{definition}
Let $(X,*,\circ)$ be a skew brace. For any oriented classical or virtual
knot or link $K$, we define the \textit{skew brace enhanced
polynomial} of $K$ to be
\[\Phi_{X}^{SB}(K)=\sum_{f\in \mathcal{C}(K,X)} u^{|\bar{Im(f)^{\circ}}|}v^{|\bar{Im(f)^{*}}|}.
\]
\end{definition}

\begin{example}
Let us illustrate the computation of the invariant for the \textit{virtual 
Hopf link} on the left and compare it with the unlink of two components on 
the right:
\[\includegraphics{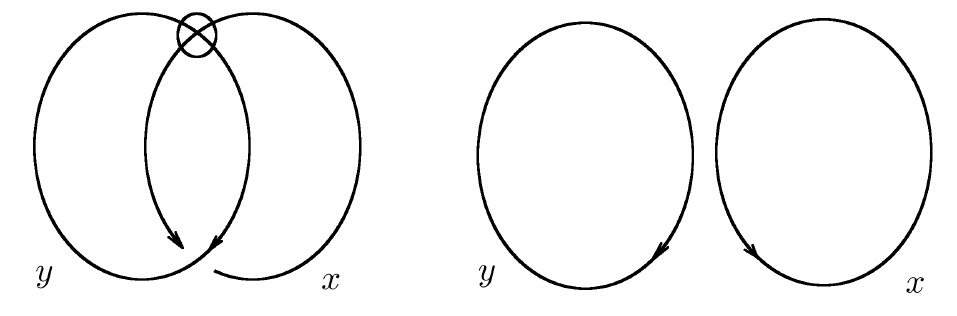}\]
Colorings of the virtual Hopf link are pairs $x,y\in X$ satisfying 
$x\utr y=x$ and $y\otr x=y$ while colorings of the unlink are just pairs
$x,y\in X$ without restriction. Then for example let $X$ be the skew brace
specified by the structure tables
\[
\begin{array}{r|rrrr}
\circ & 1 & 2 & 3 & 4 \\\hline
1 & 1 & 2 & 3 & 4 \\
2 & 2 & 3 & 4 & 1 \\
3 & 3 & 4 & 1 & 2 \\
4 & 4 & 1 & 2 & 3
\end{array}
\quad
\begin{array}{r|rrrr}
* & 1 & 2 & 3 & 4 \\\hline
1 & 1 & 2 & 3 & 4 \\
2 & 2 & 1 & 4 & 3 \\
3 & 3 & 4 & 1 & 2 \\
4 & 4 & 3 & 2 & 1.
\end{array}
\]
The coloring equations for the virtual Hopf link in terms of the skew brace
\[
\begin{array}{rcccl}
x\utr y & = & y^{\circ}\circ(x*y) & = & x \\
y\otr x & = & x^{\circ}\circ(x*y) & = & y. 
\end{array}
\]
Of the sixteen potential colorings, as the reader can verify, there are four 
which do not satisfy the conditions: $(x,y)\in\{(2,2),(2,4),(4,2),(4,4)\}$. 
For the twelve valid colorings, there are eight which have closures of the 
entire set under both group operations, three which have closures of 
cardinality two under both group operations, and one with closures of 
cardinality 1 under both group operations. Hence, we have
\[\Phi_{X}^{SB}(\mathrm{vHopf})=8u^4v^4+3u^2v^2+uv.\]
Repeating for the other colorings of the unlink, we have
\[\Phi_{X}^{SB}(\mathrm{U_2})=12u^4v^4+3u^2v^2+uv.\]
Hence the invariant detects the non-triviality of the virtual Hopf link.
\end{example}

Next, we have a definition from \cite{KSV}.

\begin{definition}
Let $X$ be a skew brace. A subset $I\subset X$ is an \textit{ideal} 
if for all $x,y\in I$ and $z\in X$ the elements
\[y^{\circ}\circ x, \quad z^**x*z,\quad z^{\circ}\circ x\circ z \quad 
\mathrm{and}\quad z^**(z\circ x) \]
are also in $I$.
\end{definition}

\begin{example}
The skew brace with structure tables
\[
\begin{array}{r|rrrrrr}
\circ & 1 & 2 & 3 & 4 & 5 & 6 \\ \hline
1 & 1 & 2 & 3 & 4 & 5 & 6 \\
2 & 2 & 3 & 4 & 5 & 6 & 1 \\
3 & 3 & 4 & 5 & 6 & 1 & 2 \\
4 & 4 & 5 & 6 & 1 & 2 & 3 \\
5 & 5 & 6 & 1 & 2 & 3 & 4 \\
6 & 6 & 1 & 2 & 3 & 4 & 5
\end{array} \quad
\begin{array}{r|rrrrrr}
* & 1 & 2 & 3 & 4 & 5 & 6 \\ \hline
1 & 1 & 2 & 3 & 4 & 5 & 6 \\
2 & 2 & 1 & 6 & 5 & 4 & 3 \\
3 & 3 & 4 & 5 & 6 & 1 & 2 \\
4 & 4 & 3 & 2 & 1 & 6 & 5 \\
5 & 5 & 6 & 1 & 2 & 3 & 4 \\
6 & 6 & 5 & 4 & 3 & 2 & 1
\end{array}
\]
has ideals including $\{1\}$, $\{1,3,5\}$ and $\{1,2,3,4,5,6\}$.
\end{example}

\begin{definition}
Let $X$ be a skew brace and $K$ an oriented knot or link represented by
a diagram $D$. Let $I(\mathrm{Im}(f))$ be the skew brace ideal generated by
the image of $f$. We define the \textit{skew brace ideal polynomial} of $K$
with respect to $X$ to be
\[\Phi_X^{I}(K)=\sum_{f\in \mathcal{C}(K,X)} u^{|I(\mathrm{Im}(f))|}.\]
\end{definition}

We can now state our main theorem.

\begin{theorem}
For any skew brace $(X,*,\circ)$, the skew brace enhanced polynomials and
skew brace ideal polynomials are invariants of oriented classical and virtual 
knots and knots and links.
\end{theorem}

\begin{proof}
The image sub-biquandle $\mathrm{Im}(f)$ of a biquandle coloring is already 
an invariant for each coloring; it follows that the contributions to the 
polynomials from each coloring are not changed by Reidemeister moves.
\end{proof}

\section{\large\textbf{Examples}}\label{E}

In this section we collect some computations and examples of the new 
invariants.

\begin{example}
Let $L$ and $L'$ be the virtual links
\[\includegraphics{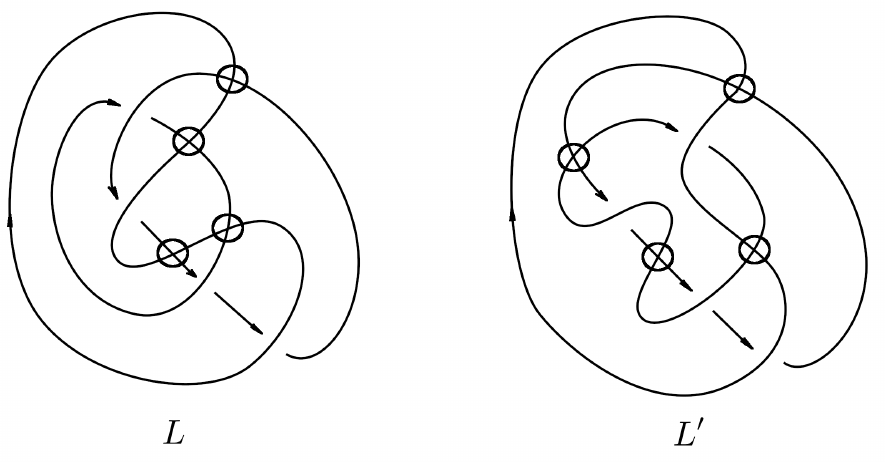}\]
and let $X$ be the skew brace with structure tables
\[
\begin{array}{r|rrrrrrrr}
\circ & 1 & 2 & 3 & 4 & 5 & 6 & 7 & 8 \\ \hline
1 & 1 & 2 & 3 & 4 & 5 & 6 & 7 & 8 \\
2 & 2 & 1 & 4 & 3 & 8 & 7 & 6 & 5 \\
3 & 3 & 4 & 1 & 2 & 6 & 5 & 8 & 7 \\
4 & 4 & 3 & 2 & 1 & 7 & 8 & 5 & 6 \\
5 & 5 & 8 & 6 & 7 & 3 & 1 & 2 & 4 \\
6 & 6 & 7 & 5 & 8 & 1 & 3 & 4 & 2 \\
7 & 7 & 6 & 8 & 5 & 2 & 4 & 3 & 1 \\
8 & 8 & 5 & 7 & 6 & 4 & 2 & 1 & 3 
\end{array}
\quad
\begin{array}{r|rrrrrrrr}
* & 1 & 2 & 3 & 4 & 5 & 6 & 7 & 8 \\ \hline
1 & 1 & 2 & 3 & 4 & 5 & 6 & 7 & 8 \\
2 & 2 & 1 & 4 & 3 & 8 & 7 & 6 & 5 \\
3 & 3 & 4 & 1 & 2 & 6 & 5 & 8 & 7 \\
4 & 4 & 3 & 2 & 1 & 7 & 8 & 5 & 6 \\
5 & 5 & 7 & 6 & 8 & 3 & 1 & 4 & 2 \\
6 & 6 & 8 & 5 & 7 & 1 & 3 & 2 & 4 \\
7 & 7 & 5 & 8 & 6 & 2 & 4 & 1 & 3 \\
8 & 8 & 6 & 7 & 5 & 4 & 2 & 3 & 1 \\
\end{array}
\]
Then our \texttt{python} computations give skew brace enhanced polynomial 
values
\[\Phi_X^{SB}(L)  =  
144u^8v^8 + 154u^4v^4 + 21u^2v^2 + uv \ne
168u^8v^8 + 130u^4v^4 + 21u^2v^2 + uv=
 \Phi_X^{SB}(L')
\] and
\[
\Phi_X^I(L)=
144u^8 + 168u^4 + 7u^2 + u\ne
168u^8 + 144u^4 + 7u^2 + u=
\Phi_X^I(L').
\]
Since both links have 320 $X$-colorings, this example shows that both
enhancements are proper and not determined by the counting invariant.
\end{example}

\begin{example}
We computed the two-variable invariant for the sets of prime 
virtual knots with up to four classical crossings 
as found at the knot atlas \cite{KA}
with respect to the skew brace $X$ with structure tables
\[ 
\begin{array}{r|rrrrrrrr}
\circ & 1 & 2 & 3 & 4 & 5 & 6 & 7 & 8 \\ \hline
1 & 1 & 2 & 3 & 4 & 5 & 6 & 7 & 8 \\
2 & 2 & 1 & 6 & 5 & 4 & 3 & 8 & 7 \\
3 & 3 & 6 & 1 & 8 & 7 & 2 & 5 & 4 \\
4 & 4 & 7 & 8 & 3 & 2 & 5 & 6 & 1 \\
5 & 5 & 8 & 7 & 6 & 1 & 4 & 3 & 2 \\
6 & 6 & 3 & 2 & 7 & 8 & 1 & 4 & 5 \\
7 & 7 & 4 & 5 & 2 & 3 & 8 & 1 & 6 \\
8 & 8 & 5 & 4 & 1 & 6 & 7 & 2 & 3
\end{array}
\begin{array}{r|rrrrrrrr}
* & 1 & 2 & 3 & 4 & 5 & 6 & 7 & 8 \\ \hline
1 & 1 & 2 & 3 & 4 & 5 & 6 & 7 & 8 \\
2 & 2 & 1 & 6 & 8 & 7 & 3 & 5 & 4 \\
3 & 3 & 6 & 1 & 7 & 8 & 2 & 4 & 5 \\
4 & 4 & 8 & 7 & 1 & 6 & 5 & 3 & 2 \\
5 & 5 & 7 & 8 & 6 & 1 & 4 & 2 & 3 \\
6 & 6 & 3 & 2 & 5 & 4 & 1 & 8 & 7 \\
7 & 7 & 5 & 4 & 3 & 2 & 8 & 1 & 6 \\
8 & 8 & 4 & 5 & 2 & 3 & 7 & 6 & 1 
\end{array}
\]
The results are collected in the table.
\[
\begin{array}{r|l}
\Phi_X^{SB}(L) & L\\ \hline
5u^2v^2+uv & 3.1, 3.2, 3.4, 4.10, 4.11, 4.15, 4.17, 4.19, 4.20, 4.22, 4.23, 4.24, 4.29, 4.32, 4.34, 4.35, \\ & 4.38, 4.39, 4.42, 4.49, 4.50, 4.57, 4.62, 4.63, 4.66, 4.67, 4.70, 4.78, 4.79 \\
2u^8v^8+5u^2v^2+uv & 2.1, 3.2, 3.5, 3.6, 3.7, 4.3, 4.6, 4.12, 4.13, 4.14, 4.18, 4.21, 4.25, 4.26, 4.27, 4.28, 4.30, \\ & 4.31, 4.36, 4.37, 4.40, 4.41, 4.43, 4.44, 4.45, 4.46, 4.47, 4.48, 4.51, 4.53, 4.54, 4.59, \\ & 4.60, 4.61, 4.64, 4.65, 4.68, 4.69, 4.71, 4.73, 4.74, 4.75,  4.80, 4.81, 4.82, 4.83, 4.84, \\ & 4.86, 4.87, 4.88, 4.91, 4.92, 4.93, 4.94, 4.95, 4.96, 4.97, 4.99, 4.100, 4.101, 4.102,\\ &  4.103,  4.104, 4.105, 4.106, 4.108 \\
6u^8v^8+5u^2v^2+uv & 4.9, 4.16, 4.33, 4.52, 4.58, 4.72 \\
14u^8v^8+5u^2v^2+uv & 4.1, 4.2, 4.4, 4.5, 4.7, 4.8, 4.55, 4.56, 4.76, 4.77, 4.85, 4.89, 4.90, 4.98, 4.107 \\
\end{array}
\]
We note that the enhancement information provides additional information beyond
the counting invariant in that it filters the colorings into classes. Each 
virtual knot in this example has a $uv$ term with coefficient 1 coming from the
monochromatic coloring by the identity element and a $u^2v^2$ term with 
coefficient 5; the differences are in the coefficients of the surjective 
$u^8v^8$ colorings, which
range from zero to 14. In particular, the virtual knots in the table with
invariant values other than $2u^8v^8+5u^2v^2+uv$ cannot be unknotted using 
crossing changes together with virtual Reidemeister moves.
\end{example}

\begin{example}
Provided $*$ is noncommutative, skew brace invariants can be effective at 
distinguishing classical knots and links as well as virtual and flat knots and
links. For example, the skew brace $X$ with noncommutative $*$ operation in 
Example \ref{ex:nab} distingishes the trefoil knot $3_1$ from the figure eight
knot $4_1$ via the counting invariant, with
\[\Phi_X^{\mathbb{Z}}(3_1)=12\ne 6=\Phi_X^{\mathbb{Z}}(4_1).\]
Our enhancements further refine this information into
\[\Phi_X^{I}(3_1)=9u^6+2u^3+u\ne 3u^6+2u^3+u =\Phi_X^{I}(4_1)\]
and\[\Phi_X^{SB}(3_1)=8u^6v^6+2u^3v^3+u^2v^2+uv\ne 2u^6v^6+2u^3v^3+u^2v^2+uv
=\Phi_X^{SB}(4_1).\]
\end{example}

\section{\large\textbf{Questions}}\label{Q}

We conclude this short paper with a list of questions for future study.

\begin{itemize}
\item What additional enhancements of the skew brace counting invariant using
the skew brace structure are possible?
\item In our examples, the powers on $u$ and $v$ are always the same; is this
true in general, or a consequence of the small cardinality of our example 
skew braces?
\end{itemize}

\bibliography{mc-sn}{}
\bibliographystyle{abbrv}

\bigskip

\noindent
\textsc{Department of Mathematical Sciences \\
Claremont McKenna College \\
850 Columbia Ave. \\
Claremont, CA 91711}

\end{document}